\documentclass{amsart}
\usepackage{indentfirst,enumerate,cite,amssymb,amsfonts,amsmath,amsthm,mathrsfs,dsfont}
\usepackage{color}
\usepackage{multicol}
\usepackage[colorlinks=true,linkcolor=blue,citecolor=blue]{hyperref}
\usepackage{enumerate}
\usepackage{geometry}
\numberwithin{equation}{section}
\usepackage{graphicx}
\newtheorem{theorem}{Theorem}[section]
\newtheorem{proposition}[theorem]{Proposition}
\newtheorem{corollary}[theorem]{Corollary}
\newtheorem{definition}[theorem]{Definition}
\newtheorem{remark}[theorem]{Remark}

\newtheorem{lemma}[theorem]{Lemma}

\newcommand{\jp}[1]{{\left\langle{#1}\right\rangle}}
\begin{document}

\title[Global Properties in Nonharmonic Analysis]{On the global properties of Fourier multipliers in the nonharmonic analysis setting}



\author[Wagner A.A. de Moraes]{Wagner Augusto Almeida de Moraes}
\address{Ghent University, 
	Department of Mathematics: Analysis, Logic and Discrete Mathematics, 
	Ghent, Belgium 
}
\email{wagneraugusto.almeidademoraes@ugent.be}
%
%



\begin{abstract}
In this paper, we investigate the global properties of Fourier multipliers in the setting of nonharmonic analysis of boundary value problems. We give necessary and sufficient conditions for a Fourier multiplier to be globally hypoelliptic and also to be globally solvable. As an application, we consider operators on $[0,1]^2$ with non-periodic boundary conditions and we obtain results that extend what is already known in the periodic case.
\end{abstract}

\maketitle

\section{Introduction}\label{sect: int}
In this paper, we are interested in the study of global properties of some boundary value problems in $\mathbb{R}^n$. Given an operator $P$ with fixed boundary conditions in a domain $\Omega \subset \mathbb{R}^n$, we want to characterize the existence as well the regularity of the solutions for the equation $Pw=f$ and to do this we will use the nonharmonic analysis of boundary value problems developed by Ruzhansky and Tokmagambetov in \cite{RT16}. This analysis is based in terms on the eigenfunctions of a model operator $L$ with the same boundary conditions in $\Omega$. This operator $L$ does not have to be either self-adjoint or elliptic and because of this, we will also work with its adjoint $L^*$ to obtain a biorthogonal system, different from the usual approach of this kind of problem. For instance, in \cite{DR18b, KM20, KMR19b} the model is obtained by a self-adjoint, elliptic and positive pseudo-differential operator. 

The study of global hypoellipticity and global solvability has been widely studied in recent years, especially in the torus $\mathbb{T}^n$ and more recently in general compact Lie groups as seen in \cite{AGKM18, Ber99, BerCorMal93, BCP04, CC00, GPY92, GW72, GW73a, GW73b, Hou79, Hou82, KMR19b, Pet11}. The Fourier analysis is present in most of the studies that deal with the question of global properties and because of this, the nonharmonic analysis of boundary value problems fits in our purpose. 

Precisely, we obtain necessary and sufficient conditions for the global properties of operators that are Fourier multipliers in the context of the model operator $L$. As an application, we study the global properties of first order operators with constant coefficients on $[0,1]^2$, where the boundary conditions are not necessarily the periodic one as in the two-dimensional torus $\mathbb{T}^2$. When the boundary condition is the periodicity we recover the classical result of Greenfield and Wallach in \cite{GW72} that relates the global hypoellipticity with Liouville numbers. In the non-periodic setting, we obtain a larger class of globally hypoelliptic operators. Still in this model, we develop the partial Fourier analysis, similarly as \cite{KMR19}, and we recover the fact that for a class of first order operator with variable coefficients, its global properties are related to the global properties of a first order operator with constant coefficient, as seen in \cite{KMR19, KMR19b}.

The paper is organized as follows. In Section \ref{nonharmonic} we present the nonharmonic analysis of boundary value problems and the classes of operators that we will study throughout the work. In Section \ref{sect: global properties} we give necessary and sufficient conditions for a Fourier multiplier to be global hypoelliptic and globally solvable. In Section \ref{sect: example} we investigate global properties in the model given by the Laplacian in $[0,1]^2$ with non necessarily periodic boundary conditions. We see that a constant coefficient first order operator fails to be global hypoelliptic only in a very specific case, which recovers the result already known in the torus $\mathbb{T}^2$. Moreover, we develop the partial Fourier analysis in $[0,1]^2$ to obtain global properties for a class of variable coefficient operator.

\section{Nonharmonic analysis of boundary value problems}\label{nonharmonic}
	
	In this section, we introduce most of the notations and preliminary results necessary for the development of this study. A very careful presentation of these concepts and the demonstration of all the results presented here can be found in \cite{RT16} and references therein. 

Let $\Omega \subset \mathbb{R}^n$ be a bounded open set and $L$ be a differential operator of order $m$ with smooth coefficients in $\Omega$, equipped with some linear bounded conditions (BC). Here, linear bounded conditions mean that the space of functions satisfying the boundary conditions is linear. Assume that $L$ has a discrete spectrum $\{ \lambda_\xi \in \mathbb{C}; \xi \in \mathcal{I}\}$, where $\mathcal{I}$ is a countable set, and $|\lambda_\xi| \to \infty$ when $|\xi| \to \infty$. We can think, without loss of generality, that $\mathcal{I}$ is a subset of $\mathbb{Z}^K$, for some $K\geq 1$. We order the eigenvalues in the ascending order
$$
|j| \leq |k| \implies |\lambda_j| \leq |\lambda_k|.
$$

Let us denote the eigenfunctions of $L$ with respect to $\lambda_\xi$ by $u_\xi$, that is,
$$
Lu_\xi = \lambda_\xi u_\xi, \quad \xi \in \mathcal{I}.
$$
Here, the functions $u_\xi$ satisfy the boundary condition (BC). The conjugate spectral problem of $L$ is
$$
L^*v_\xi = \overline{\lambda_\xi} v_\xi, \quad \xi \in \mathcal{I},
$$
equipped with the conjugate bounded condition that we will denote by (BC)*. In general, the operator $L$ does not have to be self-adjoint, neither elliptic. Throughout the text, we will obtain estimates involving the weight
$$
\jp{\xi}:= (1+|\lambda_\xi|^2)^{\frac{1}{2m}}.
$$

We can take biorthogonal systems $\{u_\xi\}_{\xi \in \mathcal{I}}$ and $\{v_\xi\}_{\xi \in \mathcal{I}}$, with $\|u_\xi\|_{L^2(\Omega)}=\|v_\xi\|_{L^2(\Omega)}=1$, for all $\xi \in \mathcal{I}$, that is,
$$
(u_\xi,v_\eta)_{L^2(\Omega)} = \int_\Omega u_\xi(x)\overline{v_\eta(x)} \, \mathrm{d} x = \left\{\begin{array}{ll}
	1,&\text{ if } \xi=\eta, \\
	0,& \text{ if } \xi \neq \eta.
\end{array} \right.
$$

We assume that $\{u_\xi;\ \xi \in \mathcal{I} \}$ is a basis for $L^2(\Omega)$, which implies that $\{v_\xi;\ \xi \in \mathcal{I}\}$ is also a basis for $L^2(\Omega)$ (see \cite{Bar51}).

The space $H^\infty_L(\overline{\Omega})$ is called the space of test functions for $L$ and is given by
$$
H^\infty_L(\overline{\Omega}):=\bigcap_{k=1}^\infty \operatorname{Dom}(L^k),
$$
where $\operatorname{Dom}\left(L^{k}\right):=\left\{f \in L^{2}(\Omega): L^{j} f \in \operatorname{Dom}(L), j=0,1,2, \ldots, k-1\right\},$ and $L^k$ is equipped with the same boundary condition (BC), for all $k \in \mathbb{N}$. The Fréchet topology of $H^\infty_L(\overline{\Omega})$ is given by the family of semi-norms
$$
\|\varphi\|_{H^k_L}:= \max_{j\leq k} \| L^j \varphi\|_{L^2(\Omega)}, \quad k \in \mathbb{N}_0, \varphi \in H^\infty_{L}(\overline{\Omega}).
$$
	
	Similarly we define the space  $H^\infty_{L^*}(\overline{\Omega})$. Notice that we have $u_\xi \in H^\infty_{L}(\overline{\Omega})$ and $v_\xi \in H^\infty_{L^*}(\overline{\Omega})$, for all $\xi \in \mathcal{I}$, which implies that both spaces $H^\infty_{L}(\overline{\Omega})$ and $H^\infty_{L^*}(\overline{\Omega})$ are dense in $L^2(\Omega)$ because we are assuming that  $\{u_\xi;\ \xi \in \mathcal{I} \}$ is a basis for $L^2(\Omega)$.
	
	For $f \in  H^\infty_{L}(\overline{\Omega})$ and $g \in  H^\infty_{L^*}(\overline{\Omega})$ we have
	$$
	(Lf,g)_{L^2(\Omega)} = (f,L^*g)_{L^2(\Omega)}.
	$$
	
	The space $H^{-\infty}_L(\Omega):=\mathcal{L}(H^\infty_{L^*}(\overline{\Omega}),\mathbb{C})$ of linear continuous functionals on $H^\infty_{L^*}(\overline{\Omega})$ is called the space of $L$-distributions. For $w \in  H^{-\infty}_{L}({\Omega})$  and $\varphi \in  H^\infty_{L^*}(\overline{\Omega})$ we write $w(\varphi) = \jp{w,\varphi}$. For any $\psi \in  H^\infty_{L}(\overline{\Omega})$, we can define
	$$
	\jp{\psi,\varphi}:=\int_\Omega \psi(x)\varphi(x)\mathrm{d}x \quad \varphi \in  H^\infty_{L^*}(\overline{\Omega}),
	$$
	which give us an embedding $ H^\infty_{L}(\overline{\Omega}) \hookrightarrow  H^{-\infty}_{L}({\Omega})$. Analogously we define the space $H^{-\infty}_{L^*}(\Omega):=\mathcal{L}(H^\infty_{L}(\overline{\Omega}),\mathbb{C})$  of $L^*$-distributions.
	
	We have the following characterization of $L$--distributions.
	\begin{proposition}
		A linear functional $w$ on $H^{\infty}_{L^*}(\overline{\Omega})$ belongs to $H^{-\infty}_L(\Omega)$ if and only if there exist $C>0$ and $k\in \mathbb{N}$ such that
		$$
		|\jp{w,\varphi}| \leq C\|\varphi\|_{H^k_{L^*}},
		$$
		for all $\varphi \in H^{\infty}_{L^*}(\overline{\Omega})$.
	\end{proposition}
	From now on, we will assume the following additional property to ensure that strongly convergent series preserve the boundary conditions:
	\begin{equation}\tag{BC+}\label{BC+}
\textrm{With $L_0$ denoting $L$ or $L^*$, if $f_j \in H^\infty_{L_0}(\overline{\Omega})$ satisfies $f_j \to f$ in $H^\infty_{L_0}(\overline{\Omega})$, then $f \in H^\infty_{L_0}(\overline{\Omega})$.}
\end{equation}

	Let $\mathcal{S}(\mathcal{I})$ be the space of rapidly decaying functions $\varphi:\mathcal{I}\to \mathbb{C}$, that is, $\varphi \in \mathcal{S}(\mathcal{I})$ if for any $M>0$ there exists a constant $C_M>0$ such that
	$$
	|\varphi(\xi)|\leq C_M\jp{\xi}^{-M}.
	$$
	
	Let $\mathcal{S}'(\mathcal{I})$ be the space of moderate growth functions $\varphi:\mathcal{I}\to \mathbb{C}$, that is, $\varphi \in \mathcal{S}'(\mathcal{I})$ if there exist constants $C,M>0$ such that
	$$
	|\varphi(\xi)|\leq C\jp{\xi}^{M}.
	$$
	We define the $L$-Fourier transform at $\xi \in \mathcal{I}$ by
	\begin{equation}\label{Fourier-def}
		\widehat{f}(\xi):= (f,v_\xi)_{L^2(\Omega)}=\int_\Omega f(x)\overline{v_\xi(x)} \mathrm{d}x.
	\end{equation}
	
	If $w$ is an $L$--distribution, we define its $L$--Fourier transform at $\xi \in \mathcal{I}$ by
	$$
	\widehat{w}(\xi):= \jp{w,\overline{v_\xi}}.
	$$
	Notice that this definition agrees with \eqref{Fourier-def} when $w$ is induced by a test function.
	
	Similarly, we define the $L^*$-Fourier transform by
\begin{equation}\label{Fourier*-def}
	\widehat{f}_*(\xi):= (f,u_\xi)_{L^2(\Omega)}=\int_\Omega f(x)\overline{u_\xi(x)} \mathrm{d}x.
\end{equation}

If $w$ is an $L^*$--distribution, we define its $L^*$--Fourier transform by
$$
\widehat{w}_*(\xi):= \jp{w,\overline{u_\xi}}.
$$

The $L$-Fourier transform is a bijective homeomorphism from $H^\infty_L(\overline{\Omega})$ to $\mathcal{S}(\mathcal{I})$ and from $H^{-\infty}_L({\Omega})$ to $\mathcal{S}'(\mathcal{I})$. The Fourier inverse formula for $f \in H^\infty_{L}(\overline{\Omega})$ is given by
$$
f(x)= \sum_{\xi \in \mathcal{I}} \widehat{f}(\xi)u_\xi(x).
$$
The $L^*$-Fourier transform is a bijective homeomorphism from $H^\infty_{L^*}(\overline{\Omega})$ to $\mathcal{S}(\mathcal{I})$ and from $H^{-\infty}_{L^*}({\Omega})$ to $\mathcal{S}'(\mathcal{I})$. The Fourier inversion formula for $f \in H^\infty_{L^*}(\overline{\Omega})$ is given by
$$
f(x)= \sum_{\xi \in \mathcal{I}} \widehat{f}_*(\xi)v_\xi(x).
$$

Hence, for $w \in H^{-\infty}_L(\Omega)$ and $\varphi \in H^\infty_{L^*}(\overline{\Omega})$ we have
$$
\jp{w,\varphi} = \jp{w, \overline{\overline{\varphi}}} = \jp{w, \overline{\sum_{\xi \in \mathcal{I}} \widehat{\overline{\varphi}}_*(\xi)v_\xi}} = \sum_{\xi \in \mathcal{I}} \overline{\widehat{\overline{\varphi}}_*(\xi)} \jp{w,\overline{v_\xi}} =  \sum_{\xi \in \mathcal{I}} \widehat{w}(\xi)\overline{\widehat{\overline{\varphi}}_*(\xi)}.
$$
The Fourier inversion formula for $w \in H^{-\infty}_L(\Omega)$ is given by
$$
w = \sum_{\xi \in \mathcal{I}} \widehat{w}(\xi)u_\xi.
$$
Precisely, for $\varphi \in H^\infty_{L^*}(\overline{\Omega})$ we have
$$
\jp{\sum_{\xi \in \mathcal{I}} \widehat{w}(\xi)u_\xi, \varphi}:= \sum_{\xi \in \mathcal{I}} \widehat{w}(\xi)\overline{\widehat{\overline{\varphi}}_*(\xi)}.
$$

Similarly, the Fourier inversion formula for  $w \in H^{-\infty}_{L^*}(\Omega)$ is given by
$$
w = \sum_{\xi \in \mathcal{I}} \widehat{w}_*(\xi)v_\xi.
$$

Although we do not have the Plancherel identity in this setting, we have the following result from Bari \cite{Bar51} that compares the $L^2$ norm and the sums of squares of Fourier coefficients.
\begin{lemma}\label{plancherel}
	There exist constants $k,K,m,M>0$ such that for every $f \in L^2(\Omega)$ we have
	$$
	m^{2}\|f\|_{L^{2}(\Omega)}^{2} \leq \sum_{\xi \in \mathcal{I}}|\widehat{f}(\xi)|^{2} \leq M^{2}\|f\|_{L^{2}(\Omega)}^{2},
	$$
	and
	$$
	k^{2}\|f\|_{L^{2}(\Omega)}^{2} \leq \sum_{\xi \in \mathcal{I}}|\widehat{f}_*(\xi)|^{2} \leq K^{2}\|f\|_{L^{2}(\Omega)}^{2}.
	$$
\end{lemma}

\begin{definition}\label{def-multiplier}
	Let $A:H^\infty_{L}(\overline{\Omega}) \to H^\infty_{L}(\overline{\Omega})$ be a continuous linear operator. We say that $A$ is an $L$-Fourier multiplier if it satisfies
	$$
	\widehat{Af}(\xi) = \sigma_A(\xi)\widehat{f}(\xi), \quad f \in H^\infty_{L}(\overline{\Omega}),
	$$
	for some $\sigma_A: \mathcal{I} \to \mathbb{C}$. We call the functions $\sigma_A$ the symbol of the operator $A$. Analogously, we say that a continuous linear operator $B:H^\infty_{L^*}(\overline{\Omega})\to H^\infty_{L^*}(\overline{\Omega})$  is an $L^*$--Fourier multiplier if it satisfies
	$$
	\widehat{Bg}_*(\xi) = \tau_B(\xi)\widehat{g}_*(\xi), \quad g \in H^\infty_{L^*}(\overline{\Omega}),
	$$
	for some $\tau_B: \mathcal{I} \to \mathbb{C}$. 
\end{definition}
Hence, for an $L$--Fourier multiplier $A: H^\infty_{L}(\overline{\Omega}) \to H^\infty_{L}(\overline{\Omega})$ we have
$$
Af(x)=\sum_{\xi \in \mathcal{I}} \sigma_A(\xi)
\widehat{f}(\xi) u_\xi(x), \quad f \in H^\infty_{L}(\overline{\Omega}),
$$
and for an $L^*$--Fourier multiplier $B: H^\infty_{L^*}(\overline{\Omega}) \to H^\infty_{L^*}(\overline{\Omega})$ we have
$$
Bg(x)=\sum_{\xi \in \mathcal{I}} \tau_B(\xi)
\widehat{g}_*(\xi) v_\xi(x), \quad g \in H^\infty_{L^*}(\overline{\Omega}).
$$
Notice that in this case we have 
\begin{equation}\label{symbol-eigen}
	Au_\xi = \sigma_A(\xi)u_\xi \text{ and } Bv_\xi = \tau_B(\xi)v_\xi,
\end{equation}
for all $\xi \in \mathcal{I}$, that is, if $A$ is an $L$--Fourier multiplier, then $u_\xi$ is an eigenfunction of $A$ with respect to the eigenvalue $\sigma_A(\xi)$ and if  $B$ is an $L^*$--Fourier multiplier, then $v_\xi$ is an eigenfunction of $B$ with respect to the eigenvalue $\tau_B(\xi)$, for all $\xi \in \mathcal{I}$.

We have the following relation between the symbols of an operator and its adjoint.
\begin{proposition}\label{symbol-adjoint}
	The operator $A:H^\infty_{L}(\overline{\Omega}) \to H^\infty_{L}(\overline{\Omega})$ is an $L$--Fourier multiplier with symbol $\sigma_A$ if and only if $A^*:H^\infty_{L^*}(\overline{\Omega}) \to H^\infty_{L^*}(\overline{\Omega})$ is an $L^*$--Fourier multiplier with symbol $\overline{\sigma_A}$.
\end{proposition}

Let $A:H^\infty_{L}(\overline{\Omega}) \to H^\infty_{L}(\overline{\Omega})$ be an $L$--Fourier multiplier. We can extend $A$ to $L$-distributions in the following way: for $w \in H^{-\infty}_L(\Omega)$ and $\varphi \in H^\infty_{L^*}(\overline{\Omega})$, define
$$
\jp{Aw,\varphi}:= \jp{w,\overline{A^*\overline{\varphi}}}.
$$
Clearly, $Aw$ is linear and the continuity follows from the continuity of $A^*$. We will still denote by $A$ its extension to $L$--distributions. 
\begin{proposition}
	Let $A$ be an $L$--Fourier multiplier. For $w \in H^{-\infty}_L({\Omega})$ we still have
	$$
	\widehat{Aw}(\xi)=\sigma_A(\xi)\widehat{w}(\xi),
	$$
	for all $\xi \in \mathcal{I}$.
\end{proposition}
\begin{proof}
	For $w \in H^{-\infty}({\Omega})$  and $\xi \in \mathcal{I}$ we have $\widehat{w}(\xi) = \jp{w, \overline{v_{\xi}}}$. Hence,
	\begin{align*}
		\widehat{Aw}(\xi) = \jp{Aw,\overline{v_\xi}} = \jp{w, \overline{A^*{v_\xi}}}.
	\end{align*}
	By \eqref{symbol-eigen} we have
	$$
	\jp{w, \overline{A^*v_\xi}} = \jp{w, \overline{\sigma_{A^*}(\xi)v_\xi}} = \overline{\sigma_{A^*}} \jp{w,\overline{v_\xi}} = \overline{\sigma_{A^*}} \widehat{w}(\xi).
	$$
	By Proposition \ref{symbol-adjoint}, we have $\sigma_{A^*} = \overline{\sigma_A}$. Therefore,
	$$
	\widehat{Aw}(\xi) =\sigma_A(\xi)\widehat{w}(\xi),
	$$
	for all $\xi \in \mathcal{I}$.
\end{proof}

\section{Global $L$--Properties}\label{sect: global properties}

In this section we will give necessary and sufficient conditions for a Fourier $L$--multiplier be globally hypoelliptic and globally solvable in the context of the nonharmonic analysis of boundary value problems. First, let us define precisely the meaning of global hypoellipticity in this setting.

\begin{definition}
	We say that an operator $P:H^{-\infty}_L(\Omega) \to H^{-\infty}_L(\Omega)$ is globally $L$-hypoelliptic if the conditions $w\in H^{-\infty}_L(\Omega)$ and $Pw \in H^\infty_L(\overline{\Omega})$ imply that $w\in H^\infty_L(\overline{\Omega})$.
\end{definition}

\begin{theorem}\label{GLH}
	Let $P$ be an $L$-Fourier multiplier with symbol $\sigma_P$. Then $P$ is globally $L$--hypoelliptic if and only if there exists $M>0$ such that
	$$
	\jp{\xi} \geq M \implies |\sigma_P(\xi)| > \jp{\xi}^{-M}.
	$$
\end{theorem}
\begin{proof}
	$(\impliedby)$ Assume that $Pw \in H^\infty_L(\overline{\Omega})$, for some $w \in H^{-\infty}_L(\Omega)$. Hence, we have
	$$
	\widehat{Pw}(\xi) = \sigma_P(\xi)\widehat{w}(\xi), \quad  \xi \in \mathcal{I}.
	$$
	By hypothesis, we have that $\sigma_P(\xi) \neq 0$ for $\jp{\xi}\geq M$ and we obtain
	$$
	|\widehat{w}(\xi)| = |\sigma_P(\xi)|^{-1} |\widehat{Pw}(\xi)| \leq \jp{\xi}^M|\widehat{Pw}(\xi)|,
	$$
	for all $\jp{\xi} \geq M$. Since $Pw \in H^{\infty}_L(\overline{\Omega})$ we have that $\widehat{Pw} \in \mathcal{S}(\mathcal{I})$ by the fact that the $L$--Fourier transform is a bijective homeomorphism from  $H^{\infty}_L(\Omega)$ to $\mathcal{S}(\mathcal{I})$. Thus, given $N>0$ there exists $C_{N+M}>0$ such that $|\widehat{Pw}(\xi)| \leq C_{N+M} \jp{\xi}^{-(N+M)}$, for all $\xi \in \mathcal{I}$. Therefore, for $\jp{\xi}\geq M$ we have
	$$
	|\widehat{w}(\xi)| \leq C_{N+M}\jp{\xi}^{-N}.
	$$  
	Notice that $\jp{\xi}<M$ only for a finitely many $\xi \in \mathcal{I}$, so for every $N>0$ there exists $C'_N>0$ such that $$
	|\widehat{w}(\xi)| \leq C'_N\jp{\xi}^{-N}, \quad \forall \xi \in \mathcal{I},
	$$   
	that is, $\widehat{w} \in \mathcal{S}(\mathcal{I})$ and we conclude that $w\in H^\infty_L(\overline{\Omega})$.
	
	$(\implies)$  Suppose that there exists a sequence $\{ \xi_k\}_{k\in \mathbb{N}}$ such that $\sigma_P(\xi_k)=0$, for every $k \in \mathbb{N}$. Consider the following function in $\mathcal{S}'(\mathcal{I})$ given by
	$$
	\alpha(\xi)=
	\left\{
	\begin{array}{ll}
		1, & \text{if } \xi=\xi_k, \text{ for some } k \in \mathbb{N};\\
		0,& \text{otherwise}.\\
	\end{array}
	\right.
	$$
	Let $w\in H^{-\infty}_L(\Omega)$ such that $\widehat{w}=\alpha$. Notice that $w\notin H^\infty_L(\overline{\Omega})$ because $\alpha \notin \mathcal{S}(\mathcal{I}$). Then
	$$
	\widehat{Pw}(\xi) = \sigma_P(\xi)\widehat{w}(\xi) = \sigma_P(\xi)\alpha(\xi)=0,
	$$
	for all $\xi \in \mathcal{I}$, which implies that $Pw=0$, so the operator $P$ is not globally $L$--hypoelliptic. 
	
	Finally, assume that for all $M>0$ there exists $\xi_M \in \mathcal{I}$ such that 
	$$
	\jp{\xi_M} \geq M \text{ and } 0<|\sigma_P(\xi_M)| < \jp{\xi_M}^{-M}.
	$$
	We can assume without loss of generality that $\jp{\xi_M}\leq \jp{\xi_{N}}$ if $M \leq N$. Consider the following function in $\mathcal{S}(\mathcal{I})$ given by
	$$
	\beta(\xi)=
	\left\{
	\begin{array}{ll}
		\sigma_P(\xi), & \text{if } \xi=\xi_k, \text{ for some } k \in \mathbb{N};\\
		0,& \text{otherwise}\\
	\end{array}
	\right.
	$$
	and the following function in $\mathcal{S}'(\mathcal{I}) \setminus \mathcal{S}(\mathcal{I})$ given by
	$$
	\alpha(\xi)=
	\left\{
	\begin{array}{ll}
		1, & \text{if } \xi=\xi_k, \text{ for some } k \in \mathbb{N};\\
		0,& \text{otherwise}.\\
	\end{array}
	\right.
	$$
	Let $g \in H^\infty_L(\overline{\Omega})$ such that $\widehat{g} = \beta$ and $w \in H^{-\infty}_L(\Omega)\setminus H^\infty_L(\overline{\Omega})$ such that $\widehat{w}=\alpha$. Notice that
	$$
	\widehat{Pw}(\xi) = \sigma_P(\xi)\widehat{w}(\xi) = \sigma_P(\xi)\alpha(w) = \beta(\xi) = \widehat{g}(\xi), 
	$$ 
	for every $\xi \in \mathcal{I}$. Hence, $Pw=g \in H^\infty_L(\overline{\Omega})$, which implies that $P$ is not globally $L$--hypoelliptic.
\end{proof}

We have the analogous result for the operator $L^*$ and since the techniques for its proof are the same as the last theorem, the proof will be omitted.
\begin{definition}
	We say that an operator $P:H^{-\infty}_{L^*}(\Omega) \to H^{-\infty}_{L^*}(\Omega)$ is globally $L^*$-hypoelliptic if the conditions $w\in H^{-\infty}_{L^*}(\Omega)$ and $Pw \in H^\infty_{L^*}(\overline{\Omega})$ imply that $w\in H^\infty_{L^*}(\overline{\Omega})$.
\end{definition}
\begin{theorem}
	Let $P$ be an $L^*$-Fourier multiplier with symbol $\tau_P$. Then $P$ is globally $L^*$--hypoelliptic if and only if there exists $M>0$ such that
	$$
	\jp{\xi} \geq M \implies |\tau_P(\xi)| > \jp{\xi}^{-M}.
	$$
\end{theorem}

Given $P$ an $L$--Fourier multiplier with symbol $\sigma_P$, we have by Proposition \ref{symbol-adjoint} that $P^*$ is an $L^*$--Fourier multiplier with symbol $\overline{\sigma_P}$. Hence, we obtain the following relation between the global hypoellipticity of $P$ and $P^*$.
\begin{corollary}
	Let $P$ be an $L$--Fourier multiplier. Then $P$ is globally $L$--hypoelliptic if and only if $P^*$ is globally $L^*$--hypoelliptic.
\end{corollary}

Now we will turn our attention to the study of the solvability of an $L$--Fourier multiplier $P$ with symbol $\sigma_P$. First, we need to characterize which $L$--distribution $f\in H^{-\infty}_L(\Omega)$ makes sense to solve the equation $Pw=f$. If $Pw=f$, for some $w \in H^{-\infty}_L(\Omega)$, then
$$
\widehat{f}(\xi) = \widehat{Pw}(\xi) = \sigma_P(\xi)\widehat{w}(\xi), \quad \xi \in \mathcal{I}.
$$ 
Hence, a necessary condition to solve the equation $Pw=f$ is that $\widehat{f}(\xi)=0$, whenever $\sigma_P(\xi)=0$. Let $\mathbb{E}$ be the space of $L$--distributions that satisfy this conditions, that is,
\begin{equation}\label{adm}
	\mathbb{E}:= \{f \in H^{-\infty}_L(\Omega); \  \sigma_P(\xi)=0 \implies \widehat{f}(\xi)=0 \}.
\end{equation}
We call the elements of $\mathbb{E}$ as $P$--admissible $L$--distributions. 
\begin{definition}
	Let $P$ be an $L$--Fourier multiplier. We say that $P$ is globally $L$--solvable if $PH^{-\infty}_L(\Omega)=\mathbb{E}$, that is, if for any $P$--admissible $L$--distribution $f\in \mathbb{E}$, there exists $w \in H^{-\infty}_L(\Omega)$ such that $Pw=f$.
\end{definition}

Similarly, if $P$ is an $L^*$--Fourier multiplier with symbol $\tau_P$, a necessary condition to solve the equation $Pw=f$ is that
$\widehat{f}_*(\xi)=0$ whenever $\tau_P(\xi)=0$. Hence, denoting by $\mathbb{E}_{*}$ the space of $P$--admissible $L^*$--distributions, we define the global $L^*$--solvability of $P$.
\begin{definition}
	Let $P$ be an $L^*$--Fourier multiplier. We say that $P$ is globally $L^*$--solvable if $PH^{-\infty}_{L^*}(\Omega)=\mathbb{E}_{*}$, that is, if for any $P$--admissible $L^*$--distribution $f\in \mathbb{E}_{*}$, there exists $w \in H^{-\infty}_{L^*}(\Omega)$ such that $Pw=f$.
\end{definition}
\begin{theorem}\label{GLS}
	Let $P$ be an $L$-Fourier multiplier with symbol $\sigma_P$. Then $P$ is globally $L$--solvable if and only if there exists $M>0$ such that
	$$
	|\sigma_P(\xi)| > \jp{\xi}^{-M},
	$$
	whenever $\sigma_P(\xi) \neq 0$.
\end{theorem}
\begin{proof}
	$(\impliedby)$ Let $f\in \mathbb{E}$. Define
	$$
	\alpha(\xi)=
	\left\{
	\begin{array}{ll}
		\sigma_P(\xi)^{-1}\widehat{f}(\xi), & \text{if } \sigma_P(\xi)\neq 0;\\
		0,& \text{otherwise}.\\
	\end{array}
	\right.
	$$
	Since $f\in H^{-\infty}_L(\Omega)$, we have that $\widehat{f} \in \mathcal{S}'(\mathcal{I})$, that is, there exist $C,K>0$ such that $|\widehat{f}(\xi)| \leq C\jp{\xi}^K$, for all $\xi \in \mathcal{I}$. When $\sigma_P(\xi) \neq 0$, we have
	$$
	|\alpha(\xi)| \leq |\sigma_P(\xi)|^{-1}|\widehat{f}(\xi)| \leq C\jp{\xi}^{M+K},
	$$
	which implies that $\alpha \in \mathcal{S}'(\mathcal{I})$. Let $w \in \mathcal{H}^{-\infty}(\Omega)$ such that $\widehat{w} = \alpha$. Hence,
	$$
	\widehat{Pw}(\xi) = \sigma_P(\xi)\widehat{w}(\xi) = \sigma_P(\xi)\alpha(\xi) = \widehat{f}(\xi), \quad \xi \in \mathcal{I}.
	$$
	Therefore $Pw=f$.
	
	$(\implies)$ Assume that there exists a sequence $\{ \xi_k\}_{k\in \mathbb{N}}$ such that 
	$$
	0 < |\sigma_P(\xi_k)| < \jp{\xi_k}^{-k}.
	$$
	Define thw following function on $\mathcal{S}'(\mathcal{I})$:
	$$
	\alpha(\xi)=
	\left\{
	\begin{array}{ll}
		1, & \text{if } \xi=\xi_k, \text{ for some } k \in \mathbb{N};\\
		0,& \text{otherwise}.\\
	\end{array}
	\right.
	$$
	Let $f \in H^{-\infty}(\Omega)$ such that $\widehat{f} = \alpha$ and notice that $f \in \mathbb{E}$. If $Pw=f$ for some $w \in H^{-\infty}(\Omega)$, then
	$$
	1 = \alpha(\xi_k)=\widehat{f}(\xi_k) = \widehat{Pw}(\xi_k) = \sigma_P(\xi_k)\widehat{w}(\xi_k),
	$$
	that is,
	$$
	|\widehat{w}(\xi_k)| = |\sigma_P(\xi_k)|^{-1} > \jp{\xi_k}^k.
	$$
	However, this implies that $\widehat{w} \notin \mathcal{S}'(\mathcal{I})$. Therefore exists $M>0$ such that $
	|\sigma_P(\xi)| > \jp{\xi}^{-M},
	$
	whenever $\sigma_P(\xi) \neq 0$, and the proof is complete.
\end{proof}
Notice that by the construction of the solution of $Pw=f$ in the proof above, we have that $\alpha \in \mathcal{S}(\mathcal{I})$ if $\widehat{f} \in \mathcal{S}(\mathcal{I})$. Hence, we obtain the following corollary.
	\begin{corollary}
		Let $P$ be an $L$--Fourier multiplier. If $P$ is globally $L$--solvable, then for any $f \in \mathbb{E} \cap H^{\infty}_L(\overline{\Omega})$, there exists $w\in H^{\infty}_L(\overline{\Omega})$ such that $Pf=w$.
	\end{corollary} 
\begin{corollary}\label{GLH->GLS}
	Let $P$ be an $L$--Fourier multiplier with symbol $\sigma_P$. If $P$ is globally $L$--hypoelliptic, then $P$ is globally $L$--solvable.
\end{corollary}
\begin{proof}
By Theorem \ref{GLH}, since $P$ is globally $L$--hypoelliptic there exists $M>0$ such that $|\sigma_P(\xi)|> \jp{\xi}^{-M}$ for every $\xi \in \mathcal{I}$ satisfying $\jp{\xi}>M$. As $\jp{\xi}<M$ only for a finitely many $\xi \in \mathcal{I}$, we obtain $M'>0$ such that $|\sigma_P(\xi)|>\jp{\xi}^{-M'}$, whenever $\sigma_P(\xi)\neq 0$. By Theorem \ref{GLS} the operator $P$ is globally $L$--solvable.
\end{proof}

The next results are about the global $L^*$--solvability of an $L^*$--Fourier multiplier and because the techniques are the same as the previous results the proofs are omitted.
	\begin{theorem}\label{GL*S}
		Let $P$ be an $L^*$-Fourier multiplier with symbol $\tau_P$. Then $P$ is globally $L^*$--solvable if and only if there exists $M>0$ such that
		$$
		|\tau_P(\xi)| > \jp{\xi}^{-M},
		$$
		whenever $\tau_P(\xi) \neq 0$. Moreover, if $P$ is globally $L^*$--solvable and $f \in \mathbb{E}_*\cap H^{-\infty}_{L^*}(\overline{\Omega})$, then there exists $w \in H^{-\infty}_{L^*}(\overline{\Omega})$ such that $Pw=f$.
	\end{theorem}

\begin{corollary}
		Let $P$ be an $L^*$--Fourier multiplier with symbol $\tau_P$. If $P$ is globally $L^*$--hypoelliptic, then $P$ is globally $L^*$--solvable.
\end{corollary}

\begin{corollary}
	Let $P$ be an $L$--Fourier multiplier with symbol $\sigma_P$. Then $P$ is globally $L$--solvable if and only if $P^*$ is globally $L^*$--solvable.
\end{corollary}
\begin{proof}
	Follows from Theorems \ref{GLS} and \ref{GL*S} by the fact that $\tau_{P^*}(\xi) = \overline{\sigma_P(\xi)}$, for all $\xi \in \mathcal{I}$.
\end{proof}

	\section{Example: Non periodic boundary conditions}\label{sect: example}
	Let $\Omega:=(0,1)^2$ and $h\in \mathbb{R}^2$ with $h>0$, that is, $h_j>0$, $j=1,2.$ Consider the operator $L_h=O^{(2)}_h$ on $\Omega$ defined by the differential operator
	$$
	O_h^{(2)}:=L_h=\frac{\partial^2}{\partial x_1^2} + \frac{\partial^2}{\partial x_2^2},
	$$
	together with the boundary conditions:
\begin{equation}\label{BC}\tag{BC}
\left.h_{j} f(x)\right|_{x_{j}=0}=\left.f(x)\right|_{x_{j}=1},\left.\quad h_{j} \frac{\partial f}{\partial x_{j}}(x)\right|_{x_{j}=0}=\left.\frac{\partial f}{\partial x_{j}}(x)\right|_{x_{j}=1}, \quad j=1,2.
\end{equation}
and the domain
$$
\operatorname{Dom}\left(L_h\right)=\left\{f \in L^2(\Omega); \ L_h f \in L^2(\Omega) \text{ and } f \text { satisfies }\eqref{BC}\right\}.
$$
Here, $\mathcal{I}=\mathbb{Z}^2$ and the system of eigenfunctions of the operator $L_h$ is
$$
\left\{u_{\xi}(x)=h^{x} \mathrm{e}^{2 \pi i x \cdot \xi}, \xi \in \mathbb{Z}^{2}\right\},
$$
with conjugate system
$$
\left\{v_{\xi}(x)=h^{-x} \mathrm{e}^{2 \pi i x \cdot \xi}, \xi \in \mathbb{Z}^{2}\right\},
$$
where $h^x:=h_1^{x_1}h_2^{x_2}$ and $x\cdot \xi = x_1\xi_1+x_2\xi_2$. Thus, we have that
$$
L_hu_{\xi}(x)=\lambda_\xi u_\xi(x), \quad L_h^*v_\xi(x)=\overline{\lambda_\xi} v_\xi(x),
$$
with
$
\lambda_\xi = (\log h_1 + 2\pi i\xi_1)^2+(\log h_2 + 2\pi i\xi_2)^2,
$
and we will denote $\jp{\xi}:=(1+|\lambda_\xi|^2)^{1/4}$.

Notice that $\overline{u_\xi}=u_{-\xi}$, $\overline{v_\xi}=v_{-\xi}$, and $\overline{\lambda_\xi} = \lambda_{-\xi}$, for all $\xi \in \mathbb{Z}^2$.

\begin{lemma}\label{equiv-jp}
	There exists $c,C>0$ such that for all $\xi \in \mathbb{Z}^2$ we have
	$$
	c\sqrt{1+|\xi|^2} \leq \jp{\xi} \leq C\sqrt{1+|\xi|^2},
	$$
	where $|\xi|^2:=|\xi_1|^2+|\xi_2|^2$.
\end{lemma}
\begin{proof}
	Notice that
	$$
	\lambda_\xi = \left[(\log h_1)^2+(\log h_2)^2 - 4\pi^2|\xi|^2\right] + i4\pi(\xi_1\log h_1+\xi_2\log h_2).
	$$
	Hence, we obtain that
\begin{align*}
|\lambda_\xi|&\leq \left|(\log h_1)^2+(\log h_2)^2 - 4\pi^2|\xi|^2\right|+4\pi|\xi_1\log h_1+\xi_2\log h_2|\\
&\leq (\log h_1)^2+(\log h_2)^2+4\pi^2|\xi|^2+|\xi_1||\log h_1|+|\xi_2||\log h_2|
\end{align*}
Since $|\xi_1|+|\xi_2| \leq |\xi|^2$, for all $\xi \in \mathbb{Z}^2$, we get a constant $C>0$ such that
\begin{equation}\label{sup-ineq}
|\lambda_\xi| \leq C(1+|\xi|^2)
\end{equation}
On the other hand, we have
$$
\left|(\log h_1)^2+(\log h_2)^2 - 4\pi^2|\xi|^2\right| \leq |\lambda_\xi|.
$$
Setting $c_1:=\min\{(\log h_1)^2+(\log h_2)^2,4\pi^2 \}$ we obtain that $c_1(|\xi|^2-1) \leq |\lambda_\xi|$, for all $\xi \in \mathbb{Z}^2\setminus \{(0,0)\}$. Here we assume that $c_1\neq 0$, because when $c_1=0$ we have that $h=(1,1)$ and the proof of the lemma is trivial in this case. Hence, there exists $c>0$ such that
\begin{equation}\label{inf-ineq}
	c(1+|\xi|^2) \leq |\lambda_\xi|,
\end{equation}
for all $|\xi| > 1$. The results follows from the expression for $\jp{\xi}$ and by the inequalities \eqref{sup-ineq} and \eqref{inf-ineq}.

\end{proof}

\begin{proposition}\label{derivative-Fourier}
	Let $f\in H^{\infty}_{L_h}(\overline{\Omega})$. We have for $j=1,2$ that
	$$
	\widehat{\frac{\partial f}{\partial x_j}}(\xi) = (\log h_j+2\pi i \xi_j) \widehat{f}(\xi).
	$$
\end{proposition}
\begin{proof}
	Let $f\in H^{\infty}_{L_h}(\overline{\Omega})$. Since $f\in \operatorname{Dom}(L_h)$, we obtain that $f(x)\overline{v_\xi(x)}\bigg|_{x_j=0}^1=0$. Hence, by integration by parts we obtain
	$$
	\widehat{\frac{\partial f}{\partial x_j}}(\xi) =\int_\Omega \frac{\partial f}{\partial x_j}(x) \overline{v_\xi(x)} \mathrm{d} x =  - \int_\Omega f(x) \overline{\frac{\partial v_\xi}{\partial x_j}(x)} \mathrm{d} x. 
	$$
	Therefore,
	$$
	\widehat{\frac{\partial f}{\partial x_j}}(\xi)= - \int_\Omega f(x) \overline{\frac{\partial v_\xi}{\partial x_j}(x)} \mathrm{d} x=(\log h_j+2\pi i \xi_j)\int_\Omega f(x) \overline{v_\xi(x)} \mathrm{d} x = (\log h_j+2\pi i \xi_j)\widehat{f}(\xi).
	$$
\end{proof}
\subsection{Constant coefficient case}\label{subsect: constant}

Consider the operator $P:H^{-\infty}_{L_h}(\Omega) \to H^{-\infty}_{L_h}(\Omega)$ given by
$$
P=\frac{\partial}{\partial{x_1}}+c\frac{\partial}{\partial{x_2}},
$$
with $c \in \mathbb{C}\setminus\{0\}$.
When $h=(1,1)$ we have periodic boundary conditions and the analysis considered in this paper coincides with the classical toroidal calculus. In particular, we have the following characterization of the global properties of the operator $P$ on the torus $\mathbb{T}^2$.
\begin{definition}
	A Liouville number is a real number $y$ with the property that, for every positive integer $n$, there exists infinitely many pairs of integers $(p,q)$, with $q>1$ such that
	$$
	0<\left|y-\frac{p}{q}\right| < \frac{1}{q^n}.
	$$
\end{definition}
\begin{theorem}[Greenfield-Wallach, \cite{GW72}; Hounie, \cite{Hou79}] The operator $P:\mathcal{D}'(\mathbb{T}^2) \to \mathcal{D}'(\mathbb{T}^2)$ is globally hypoelliptic if and only if either $\operatorname{Im} c \neq 0$ or $c$ is an irrational non-Liouville number. Moreover, the operator  $P:\mathcal{D}'(\mathbb{T}^2) \to \mathcal{D}'(\mathbb{T}^2)$ is globally solvable if and only if either $\operatorname{Im} c \neq 0$, or $c$ is a rational number, or $c$ is an irrational non-Liouville number.
\end{theorem} 
By Proposition \ref{derivative-Fourier} we have
$$
\widehat{Pw}(\xi)=(\log h_1 + 2\pi i\xi_1 + c(\log h_2 + 2\pi i\xi_2) )\widehat{w}(\xi).
$$
Setting $c=a+ib$, we obtain
$$
\widehat{Pw}(\xi)=\sigma_P(\xi)\widehat{w}(\xi),
$$
where
\begin{equation}\label{symbol-P}
	\sigma_P(\xi)=(\log (h_1h_2^a) - 2\pi b\xi_2) + i(\log h_2^b+2\pi(\xi_1 + a\xi_2)),
\end{equation}
that is, $P$ is an ${L_h}$--Fourier multiplier with symbol $\sigma_P$. Let us see in the next results that the global properties of the operator $P$ depends of the choice of $c$ as well as $h$, that is, the environment where we are working.

\begin{theorem}
	Let $L_h=O^{(2)}_h$, $c\in \mathbb{C}\setminus\{0\}$, and assume that $c \log h_2 \neq -\log h_1$. Then the operator  $P:H^{-\infty}_{L_h}(\Omega) \to H^{-\infty}_{L_h}(\Omega)$ given by
	$$
	P=\frac{\partial}{\partial{x_1}}+c\frac{\partial}{\partial{x_2}}
	$$ 
	is globally ${L_h}$-hypoelliptic.
\end{theorem}
\begin{proof}
	By the hypothesis $c \log h_2 \neq -\log h_1$ we have two cases to consider:
	
	\textsc{{\textbf Case I:}} $h_1h_2^a\neq 1$
	
	In this case, we have that $\log(h_1h_2^a) \neq 0$, which implies that $\sigma_P(\xi)$ vanishes for at most one $\xi \in \mathbb{Z}^2$. Moreover, there exists $C>0$ such that
	$$
	|\sigma_P(\xi)| \geq \frac{1}{C},
	$$
	whenever $\sigma_P(\xi)\neq0$. Hence, we obtain that
	$$
	|\widehat{w}(\xi)| \leq C \left|\widehat{Pw}(\xi)\right|,
	$$
	whenever $\sigma_P(\xi) \neq 0$. In particular, we obtain that $\widehat{w} \in \mathcal{S}(\mathbb{Z}^2)$, whenever $\widehat{Pw} \in \mathcal{S}(\mathbb{Z}^2)$ . Therefore the operator $P$ is globally ${L_h}$-hypoelliptic.
	
	\textsc{{\textbf Case II:}} $h_1h_2^a= 1$ and $b \neq 0$.
	
	Since $b \neq 0$, we have that $\mathrm{Re}(\sigma_P(\xi))=0$ only when $\xi_2=0$. Again we obtain that  $\sigma_P(\xi)$ vanishes for at most one $\xi \in \mathbb{Z}^2$ and that there exists $C>0$ such that
	$
	|\sigma_P(\xi)| \geq \frac{1}{C},
	$
	whenever $\sigma_P(\xi)\neq0$. The conclusion follows as the previous case.
\end{proof}

The next corollary follows directly from Corollary \ref{GLH->GLS}.
\begin{corollary}
	Let $L_h=O^{(2)}_h$, $c\in \mathbb{C}\setminus\{0\}$, and assume that $c \log h_2 \neq -\log h_1$. Then the operator  $P:H^{-\infty}_{L_h}(\Omega) \to H^{-\infty}_{L_h}(\Omega)$ given by
	$$
	P=\frac{\partial}{\partial{x_1}}+c\frac{\partial}{\partial{x_2}}
	$$ 
	is globally ${L_h}$-solvable.
\end{corollary}
Notice that the hypothesis  $c \log h_2 \neq -\log h_1$ implies that $h \neq (1,1)$, that is, we are not in the torus setting. By the previous theorem, when we are not working in the torus, the operator $P$ is almost always globally ${L_h}$-hypoelliptic, and the only $c\in \mathbb{C}\setminus\{0\}$ that we have to analyze separately is the one that satisfies $c\log h_2 = -\log h_1$. The next theorem gives us necessary and sufficient conditions for global ${L_h}$-properties of the operator $P$ for this case.

\begin{theorem}
	Assume that $c\log h_2 = -\log h_1$. Then the operator 
	$$
	P=\frac{\partial}{\partial{x_1}}+c\frac{\partial}{\partial{x_2}}
	$$
	is globally ${L_h}$-hypoelliptic if and only if either $\mathrm{Im} (c)\neq 0$, or $c$ is an irrational non-Liouville number. Moreover, the operator $P$ is globally ${L_h}$-solvable if and only if either $\mathrm{Im} (c)\neq 0$, or $c$ is a rational number, or $c$ is an irrational non-Liouville number.
\end{theorem}
\begin{proof}
	If $h=(1,1)$ then the hypothesis $c\log h_2 = -\log h_1$ holds for every $c\in \mathbb{C}$ and the result was proven by Greenfield and Wallach in \cite{GW72}.
	
	If $h\neq(1,1)$, then $\mathrm{Im} (c)=b=0$ and we have $c = a=-\frac{\log h_1}{\log h_2}$. Hence,
	$$
	\sigma_P(\xi) = i2\pi(\xi_1+a\xi_2).
	$$
	
	By Lemma \ref{equiv-jp} the analysis of this case is the same as the torus case, so by Theorem \ref{GLH} the operator $P$ is globally  $L_h$--hypoelliptic if and only if $c$ is an irrational non-Liouville number and by Theorem \ref{GLS} the operator $P$ is globally $L_h$--solvable if and only if either $c$ is a rational number, or $c$ is an irrational non-Liouville number.
\end{proof}

\begin{remark}
	Since our results are based on the bijection from $H_{L_h}^\infty(\Omega)$ to $\mathcal{S}(\mathbb{Z}^2)$ given by the ${L_h}$--Fourier transform, we obtain the same results replacing ${L_h}$ by ${L_h}^*$, because the ${L_h}^*$--Fourier transform is also a bijection from $H_{{L_h}^*}^\infty(\Omega)$ to $\mathcal{S}(\mathbb{Z}^2)$.
\end{remark}
\subsection{Variable coefficient case}\label{subsect: variable}
When $h=(1,1)$, the model described in Section \ref{nonharmonic} recover the classical harmonic analysis on the two-dimensional torus, that is, we can identify $H^\infty_{L_{(1,1)}}(\overline{\Omega})\simeq C^\infty(\mathbb{T}^2)$. We have the following property holding for any $h>0$.
\begin{proposition}
	Let ${L_h}=O^{(2)}_h$, with $h\in \mathbb{R}^2$ satisfying $h_1,h_2>0$. Then the space $H^\infty_{L_h}(\overline{\Omega})$ is a $C^\infty(\mathbb{T}^2)$--module.
\end{proposition}
\begin{proof}
	Let us proof that for any $f\in H^\infty_{L_h}(\overline{\Omega})$ and $g \in C^\infty(\mathbb{T}^2)$ we have that $fg \in H^\infty_{L_h}(\overline{\Omega})$. Indeed, we have that $fg \in L^2(\Omega)$ and $L_h^k(fg) \in L^2(\Omega)$, for any $k\in \mathbb{N}$. Let us check that $fg$ satisfies \eqref{BC}:
	\begin{itemize}
		\item $h_j(fg)(x)\big|_{x_j=0} = \left(h_jf(x)\big|_{x_j=0}\right)g(x)\big|_{x_j=0} =   f(x)\big|_{x_j=1}g(x)\big|_{x_j=1}=(fg)(x)\big|_{x_j=1}$.
			\item $\begin{aligned}h_j\dfrac{\partial(fg)}{\partial x_j}(x)\big|_{x_j=0} &=h_j\left[\dfrac{\partial f}{\partial x_j}(x)g(x)\big|_{x_j=0}+f(x)\dfrac{\partial g}{\partial x_j}(x)\big|_{x_j=0}\right]\\ &=\left[\dfrac{\partial f}{\partial x_j}(x)g(x)\big|_{x_j=1}+f(x)\dfrac{\partial g}{\partial x_j}(x)\big|_{x_j=1}\right]\\&= \dfrac{\partial(fg)}{\partial x_j}(x)\big|_{x_j=1}\end{aligned}$.
	\end{itemize}
Proceeding analogously we conclude that $L_h^k(fg)$ satisfies the boundary condition \eqref{BC} for every $k \in \mathbb{N}$. Therefore $fg \in H^\infty_{L_h}(\overline{\Omega})$.
\end{proof}

Let $a \in C^\infty(\mathbb{T}^1)$ be a real-valued function and consider the operator $P:H_{L_h}^{-\infty}(\Omega) \to H^{-\infty}_{L_h}(\Omega)$ given by
$$
P=\frac{\partial}{\partial x_1}+ a(x_1)\frac{\partial}{\partial x_2}.
$$
The operator $P$ is well-defined because we can see $a$ as an element $\tilde{a} \in C^\infty(\mathbb{T}^2)$ satisfying $\tilde{a}(x_1,x_2)=a(x_1)$, for every $x_2 \in [0,1]$.

\begin{definition}
	Let $f\in L^1(\Omega)$ and $\xi_1 \in \mathbb{Z}$. We define the Partial ${L_h}$--Fourier coefficient of $f$ with respect to $x_1$ at $\xi_1$ as the function $\mathcal{F}_j f (\xi_1,x_2)$ given by
	$$
	\mathcal{F}_1 f (\xi_1,x_2)= \int_0^1 f(x_1,x_2)\overline{v_{\xi_1}(x_1)} \mathrm{d}x_1,
	$$
	where $v_{\xi_j}(x_j):= h_j^{-x_j}e^{2\pi i x_j\xi_j}$, for $j=1,2$. Similarly we define the Partial ${L_h}$--Fourier coefficient of $f$ with respect to $x_2$ at $\xi_2\in \mathbb{Z}$ by
	$$
	\mathcal{F}_2 f (x_1,\xi_2)= \int_0^1 f(x_1,x_2)\overline{v_{\xi_2}(x_2)} \mathrm{d}x_2.
	$$
\end{definition}
Notice that $v_\xi(x)=v_{\xi_1}(x_1)v_{\xi_2}(x_2)$. So, we have
$$
\widehat{f}(\xi_1,\xi_2) = \int_0^1  \mathcal{F}_1 f (\xi_1,x_2)\overline{v_{\xi_2}(x_2)} \mathrm{d}x_2 = \int_0^1 	\mathcal{F}_2 f (x_1,\xi_2)\overline{v_{\xi_1}(x_1)} \mathrm{d}x_1.
$$
Moreover, we can write
\begin{equation}\label{transform-partial}
f(x_1,x_2) = \sum_{\xi_1 \in \mathbb{Z}}\mathcal{F}_1f(\xi_1,x_2)u_{\xi_1}(x_1) = \sum_{\xi_2 \in \mathbb{Z}}\mathcal{F}_2f(x_1,\xi_2)u_{\xi_2}(x_2)
\end{equation}
We have also the analogous for the adjoint operator ${L_h}^*$.
\begin{definition}
	Let $f\in L^1(\Omega)$ and $\xi_1 \in \mathbb{Z}$. We define the Partial ${L_h}^*$--Fourier coefficient of $f$ with respect to $x_1$ at $\xi_1$ as the function $\mathcal{F}^*_j f (\xi_1,x_2)$ given by
	$$
	\mathcal{F}^*_1 f (\xi_1,x_2)= \int_0^1 f(x_1,x_2)\overline{u_{\xi_1}(x_1)} \mathrm{d}x_1,
	$$
	where $u_{\xi_j}(x_j):= h_j^{x_j}e^{2\pi i x_j\xi_j}$, for $j=1,2$. Similarly we define the Partial ${L_h}^*$--Fourier coefficient of $f$ with respect to $x_2$ at $\xi_2\in \mathbb{Z}$ by
	$$
	\mathcal{F}^*_2 f (x_1,\xi_2)= \int_0^1 f(x_1,x_2)\overline{u_{\xi_2}(x_2)} \mathrm{d}x_2.
	$$
\end{definition}
Notice that $u_\xi(x)=u_{\xi_1}(x_1)u_{\xi_2}(x_2)$. So, we have
$$
\widehat{f}_*(\xi_1,\xi_2) = \int_0^1  \mathcal{F}_1^* f (\xi_1,x_2)\overline{u_{\xi_2}(x_2)} \mathrm{d}x_2 = \int_0^1 	\mathcal{F}_2^* f (x_1,\xi_2)\overline{u_{\xi_1}(x_1)} \mathrm{d}x_1.
$$
Moreover, we can write
$$
f(x_1,x_2) = \sum_{\xi_1 \in \mathbb{Z}}\mathcal{F}_1^*f(\xi_1,x_2)v_{\xi_1}(x_1) = \sum_{\xi_2 \in \mathbb{Z}}\mathcal{F}^*_2f(x_1,\xi_2)v_{\xi_2}(x_2)
$$

Consider the operator $L_j =-i \dfrac{\mathrm{d}}{\mathrm{d} x_j}$ on $(0,1)$, $j=1,2$, with boundary condition
\begin{equation}\label{BC-1d}
h_jf(0)=f(1),
\end{equation}
and domain $\operatorname{Dom}(L_j)=\{f \in L^2((0,1)); L_jf \in L^2((0,1)) \text{ and } f \text{ satisfies } \eqref{BC-1d} \}$. With $\mathcal{I}=\mathbb{Z}$, the system of eigenfunctions of the operator $L_j$ is $\{u_{\xi_j}(x_j)=h_j^{x_j}e^{2\pi ix_j\xi_j} \}$ with eigenvalues $\lambda_{\xi_j} = -i\log h_j+2\pi \xi_j$, and the conjugate system is $\{v_{\xi_j}(x_j)=h_j^{-x_j}e^{2\pi ix_j\xi_j}\}$. Here, the weight is given by  $\jp{\xi_j} = (1+|\lambda_{\xi_j}|^2)^{1/2}$ and we have the following relation between the weights of the model $L_h$ and $L_j$.
\begin{proposition}
	Let $\xi \in \mathbb{Z}^2$. There exists $k,K>0$ such that
\begin{equation}\label{jp-bracket}
	k(\jp{\xi_1}+\jp{\xi_2}) \leq \jp{\xi} \leq K(\jp{\xi_1}+\jp{\xi_2})
\end{equation}
\end{proposition}
\begin{proof}
	Notice that $\lambda_\xi = (\log h_1 + 2\pi i\xi_1)^2+(\log h_2 + 2\pi i\xi_2)^2 = -(\lambda_{\xi_1}^2 + \lambda_{\xi_2}^2)$.	Since $\jp{\xi}=(1+|\lambda_{\xi}|^2)^{1/4}$ and $\jp{\xi_j} = (1+|\lambda_{\xi_j}|^2)^{1/2}$, we have that
	$$
\lim_{|\xi|\to \infty} \frac{\jp{\xi}}{\jp{\xi_1}+\jp{\xi_2}} = 1,
	$$
so we can get $k,K>0$ satisfying \eqref{jp-bracket}.
\end{proof}
\begin{definition}
	Let $w \in H^{-\infty}_{L_h}(\Omega)$ and $\xi_1 \in \mathbb{Z}$. We define the partial ${L_h}$-Fourier coefficient of $w$ with respect to $x_1$ at $\xi_1$ as the linear functional $\mathcal{F}_1w(\xi_1,\cdot): H^{\infty}_{L^*_2}((0,1)) \to \mathbb{C}$ given by
	$$
	\jp{\mathcal{F}_1w(\xi_1,\cdot), \varphi} := \jp{w, v_{-\xi_1} \times \varphi},
	$$
	where $(v_{-\xi_1} \times \varphi)(x_1,x_2):=v_{-\xi_1}(x_1) \varphi(x_2)$, for all $\varphi \in H^{\infty}_{L^*_2}((0,1))$.
	Similarly, we define the partial ${L_h}$-Fourier coefficient of $w$ with respect to $x_2$ at $\xi_2 \in \mathbb{Z}$ as the linear functional $\mathcal{F}_2w(\cdot,\xi_2): H^{\infty}_{L^*_1}((0,1)) \to \mathbb{C}$ given by
	$$
	\jp{\mathcal{F}_2w(\cdot,\xi_2), \varphi} := \jp{w,\varphi \times v_{-\xi_2}},
	$$
	where $(\varphi \times v_{-\xi_2})(x_1,x_2):=\varphi(x_1)v_{-\xi_2}(x_2) $, for all $\varphi \in H^{\infty}_{L^*_1}((0,1))$.
\end{definition}

For $w \in H^{-\infty}_{L_h}(\Omega)$, we may write
$$
w = \sum_{\xi_1 \in \mathbb{Z}}\mathcal{F}_1w(\xi_1,\cdot)u_{\xi_1},
$$
where for $\varphi \in H^\infty_{{L_h}^*}(\Omega)$ we have
$$
\jp{w,\varphi}= \jp{\sum_{\xi_1 \in \mathbb{Z}}\mathcal{F}_1w(\xi_1,\cdot)u_{\xi_1},\varphi} := \sum_{\xi_1 \in \mathbb{Z}}\jp{\mathcal{F}_1w(\xi_1,\cdot), \mathcal{F}_1^*\varphi(-\xi_1, \cdot)}.
$$
{
The idea here is to use the partial Fourier theory to construct an automorphism $\Psi_a$ of $H^{-\infty}_{L_h}(\Omega)$, where the restriction to $H^\infty_{L_h}(\Omega)$ remains an automorphism of $H^\infty_{L_h}(\Omega)$, satisfying $\Psi_a P = P_0 \Psi_a$, where $P_0 = \partial_{x_1} +a_0\partial_{x_2}$, and $a_0=\int a(s) \mathrm{d} s$.

\begin{proposition}
When $f\in H^\infty_{L_h}(\Omega)$ we have that $\mathcal{F}_1f(\xi_1, \cdot) \in H^\infty_{L_2}((0,1))$ and $\mathcal{F}_2f(\cdot, \xi_2)\in H^\infty_{L_1}((0,1))$, for all $\xi_1,\xi_2 \in \mathbb{Z}$. When $w\in H^{-\infty}_{L_h}(\Omega)$ we have that $\mathcal{F}_1w(\xi_1, \cdot) \in H^{-\infty}_{L_2}((0,1))$ and $\mathcal{F}_2w(\cdot, \xi_2)\in H^{-\infty}_{L_1}((0,1))$, for all $\xi_1,\xi_2 \in \mathbb{Z}$.
\end{proposition}
\begin{proof}
Let $f\in H^\infty_{{L_h}}(\Omega)$ and $\xi_1 \in \mathbb{Z}$. First, let us show that $\mathcal{F}_1f(\xi_1,\cdot) \in L^2((0,1))$. We have
\begin{align*}
	\int_0^1 \left|\mathcal{F}_1f(\xi_1,x_2)\right|^2 \mathrm{d}x_2 &= \int_0^1\left|\int_0^1 f(x_1,x_2) v_{\xi_1}(x_1) \mathrm{d}x_1\right|^2 \mathrm{d}x_2 \\
	&\leq \int_0^1 \left(\int_0^1 |f(x_1,x_2)v_{\xi_1}(x_1)| \mathrm{d}x_1\right)^2 \mathrm{d}x_2 \\
	&\leq \int_0^1 \left( \int_0^1 |f(x_1,x_2)|^2 \mathrm{d}x_1 \int_0^1 |v_{\xi_1}(x_1)|^2 \mathrm{d}x_1\right) \mathrm{d}x_2 \\
	&=\|f\|_{L^2(\Omega)}^2. 
\end{align*}
We also have that $\mathcal{F}_1f(\xi_1,\cdot)$ satisfies the boundary condition of the operator $L_2$. Indeed,
$$
h_2 \mathcal{F}_1f(\xi_1,0) = \int_0^1 h_2f(x_1,0) v_{\xi_1}(x_1) \mathrm{d}x_1 = \int_0^1 f(x_1,1) v_{\xi_1}(x_1) \mathrm{d}x_1 =\mathcal{F}_1f(\xi_1,1). 
$$
	Notice that for any $k \in \mathbb{N}$ we have that $L_2^k \mathcal{F}_1f(\xi_1,\cdot) = (-i)^k \mathcal{F}_1 \frac{\partial^k f}{\partial x_2^k}(\xi_1,\cdot)$, and since $f\in H^\infty_{L_h}(\Omega)$
we conclude that $\mathcal{F}_1f(\xi_1,\cdot) \in H^{\infty}_{L_2}((0,1))$.

Let $w \in H^{-\infty}_{L_h}(\Omega)$. So there exists $C>0$ and $k\in \mathbb{N}$ such that
$$
|\jp{w,\psi}| \leq C\|\psi\|_{H^k_L}, \quad \psi \in H^\infty_{{L_h}^*}(\Omega),
$$
where $\|\psi\|_{H^k_{{L_h}^*}}= \max\limits_{j\leq k}\|({L_h}^*)^j\psi\|_{L^2(\Omega)}$. Hence, for all $\varphi \in H^{\infty}_{L_2}((0,1))$ we have
$$
\left|\jp{\mathcal{F}w(\xi_1,\cdot), \varphi}\right| = |\jp{w, v_{-\xi_1}\times \varphi}| \leq C\|v_{-\xi}\times \varphi\|_{H^k_{{L_h}^*}} \leq C'\|\varphi\|_{H^{2k}_{L^*_2}},
$$
for some $C'>0$, which implies that $\mathcal{F}_1w(\xi_1,\cdot)\in H^{-\infty}_{L_2}((0,1))$. The proofs for the partial $L$-Fourier transform with respect to the second variable is analogous, as well as for the partial $L^*_h$--Fourier coefficient and will be omitted.
\end{proof}
\begin{theorem}\label{Partial-Func}
	We have that $f \in H_{L_h}^\infty(\Omega)$ if and only if $\mathcal{F}_2f(\cdot, \xi_2)\in H^\infty_{L_1}((0,1))$, for all $\xi_2 \in \mathbb{Z}$, and for every $k\in \mathbb{N}$ and $N>0$, there exists $C_{kN}>0$ such that
	$$
\left\|\frac{\mathrm{d^k}}{\mathrm{d}x_1^k}\mathcal{F}_2f(x_1, \xi_2)  \right\|_{L^2((0,1))} \leq C_{kN}\jp{\xi_2}^{-N}, \quad \forall \xi_2 \in \mathbb{Z}.
	$$
\end{theorem} 
\begin{proof}
	$(\impliedby)$ By the Plancherel inequality given in Lemma \ref{plancherel} and by the relation \eqref{transform-partial}, for each fixed $\xi_2 \in \mathbb{Z}$ we have
	$$
	\sum_{\xi_1\in\mathbb{Z}} |\lambda_{\xi_1}^k\widehat{f}(\xi_1,\xi_2)|^2 = \sum_{\xi_1\in\mathbb{Z}} \left|\widehat{\frac{\partial^k f}{\partial x_1^k}}(\xi_1,\xi_2)\right|^2 \leq M^2\left\|\frac{\mathrm{d}^k}{\mathrm{d}x_1^k}\mathcal{F}_2f(\cdot, \xi_2)\right\|_{L^2((0,1))}^2 \leq C_{kN}^2 \jp{\xi_2}^{-2N}
	$$
	In particular, for $k=N$, we have for some $C_N>0$ that
	$$
	|\lambda_{\xi_1}^N\widehat{f}(\xi)| \leq C_N \jp{\xi_2}^{-N},
	$$
	for all $\xi_1 \in \mathbb{Z}$. Hence, for $\lambda_{\xi_1}\neq0$ we obtain
	$$
		|\widehat{f}(\xi)| \leq C_N (\jp{\xi_2}\jp{\xi_1})^{-N} \leq C_N2^N (\jp{\xi_2}+\jp{\xi_1})^{-N} \leq C_N'\jp{\xi}^{-N},
	$$
	which implies that $\widehat{f} \in \mathcal{S}(\mathbb{Z}^2)$ and then $f \in H^\infty_{L_h}(\Omega)$.
	
	$(\implies)$ Since $f\in H^\infty_{L_h}(\Omega)$, for any $k,N \in \mathbb{N}$ we have $g=\dfrac{\partial^{k+N}}{\partial x_1^k \partial x_2^N}f \in L^2(\Omega)$. Thus, for each $\xi_2\in \mathbb{Z}$ we have  $\mathcal{F}_2 g(\cdot,\xi_2)\in L^2((0,1))$, and
	$$
	\left\|\lambda_{\xi_2}^N\frac{\mathrm{d^k}}{\mathrm{d}x_1^k}\mathcal{F}_2f(\cdot,\xi_2)\right\|_{L^2((0,1))}^2 = \| \mathcal{F}_2g(\cdot,\xi_2)\|_{L^2((0,1))}^2 \leq m^2  \sum_{\xi_1\in \mathbb{Z}} |\widehat{g}(\xi_1,\xi_2)|^2 \leq  m^2\sum_{\xi\in \mathbb{Z}^2} |\widehat{g}(\xi)|^2\leq  C\|g\|_{L^2(\Omega)}^2
	$$
	Therefore, for any $k,N \in \mathbb{N}$, there exists $C_{kN}>0$ such that
	$$
	\left\|\frac{\mathrm{d^k}}{\mathrm{d}x_1^k}\mathcal{F}_2f(x_1, \xi_2)  \right\|_{L^2((0,1))} \leq C_{kN}\jp{\xi_2}^{-N},
	$$
	for all $\xi_2 \in \mathbb{Z}$ and the proof is complete.
\end{proof}
\begin{theorem}
	We have that $w \in H_{L_h}^{-\infty}(\Omega)$ if and only if $\mathcal{F}_2w(\cdot, \xi_2)\in H^{-\infty}_{L_1}((0,1))$, for all $\xi_2 \in \mathbb{Z}$, and there exist $K\in \mathbb{N}$ and $C>0$ such that
	$$
	\left| 	\jp{\mathcal{F}_2w(\cdot,\xi_2), \varphi} \right| \leq C p_K(\varphi)\jp{\xi_2}^{K},
	$$
	for all $\varphi \in H^\infty_{L^*_1}((0.1))$, where $p_K(\varphi) := \sum\limits_{\beta \leq K} \left\|\frac{\mathrm{d}^\beta\varphi}{\mathrm{d}x^\beta}\right\|_{L^2((0,1))}$.
\end{theorem} 
\begin{proof}
$(\impliedby)$ For $\xi_1\in\mathbb{Z}$, take $\varphi = v_{-\xi_1}$. Hence,
$$
|\widehat{w}(\xi)| = \jp{\mathcal{F}_2w(\cdot, \xi_2), v_{-\xi_1}} \leq Cp_K(v_{-\xi_1})\jp{\xi_2}^K.
$$
Notice that
$$
p_K(v_{-\xi_1}) =  \sum\limits_{\beta \leq K} \left\|\frac{\mathrm{d}^\beta v_{-\xi_1}}{\mathrm{d}x^\beta}\right\|_{L^2((0,1))} =  \sum\limits_{\beta \leq K}|\lambda_{-\xi_1}|^\beta \left\| v_{-\xi_1}\right\|_{L^2((0,1))} \leq C\jp{\xi_1}^K
$$
Therefore, there exist $C,K>0$ such that $|\widehat{w}(\xi)|\leq C\jp{\xi}^K$, for all $\xi \in \mathbb{Z}^2$, which implies that $\widehat{w} \in \mathcal{S}'(\mathbb{Z}^2)$, and consequently $w \in H^{-\infty}_{L_h}(\Omega)$.

$(\implies)$ Since $w \in H^{-\infty}_{L_h}(\Omega)$, there exist $C,K>0$ such that $|\widehat{w}(\xi)| \leq C\jp{\xi}^K$, for all $\xi \in \mathbb{Z}^2$ and we may write
$$
w=\sum_{\xi \in \mathbb{Z}^2}\widehat{w}(\xi)u_{\xi} = \sum_{\xi_1,\xi_2 \in \mathbb{Z}}\widehat{w}(\xi_1,\xi_2)u_{\xi_1}u_{\xi_2}
$$

For $\varphi \in H^{\infty}_{L^*_1}((0,1))$ we have
$$
	\left| 	\jp{\mathcal{F}_2w(\cdot,\xi_2), \varphi} \right| = \left| \jp{w,\varphi \times v_{-\xi_2}}\right| =\left| \sum_{\xi_1,\eta \in \mathbb{Z}}\widehat{w}(\xi_1,\eta)\jp{u_{\xi_1}, \varphi} \jp{u_{\eta},v_{-\xi_2}} \right| 
$$
By the biorthogonality of the systems $\{u_\xi\}$ and $\{v_\xi\}$, we have that $\jp{u_{\eta},v_{-\xi_2}}=(u_{\eta},v_{\xi_2})_{L^2((0,1))}=\delta_{\eta \xi_2}$, there $\delta_{mn}$ is the Kronecker's delta. Moreover, we have $\jp{u_{\xi_1},\varphi} = \widehat{\varphi}_*(-\xi_1)$. Thus,
$$
\left| 	\jp{\mathcal{F}_2w(\cdot,\xi_2), \varphi} \right| \leq \sum_{\xi_1\in\mathbb{Z}}|\widehat{w}(\xi_1,\xi_2)||\widehat{\varphi}_* (-\xi_1)| \leq C\sum_{\xi_1\in \mathbb{Z}} \jp{\xi}^K|\widehat{\varphi}_* (-\xi_1)| \leq C\jp{\xi_2}^K\sum_{\xi_1 \in \mathbb{Z}}\jp{\xi_1}^K|\widehat{\varphi}_* (\xi_1)|,
$$
where the last inequality follows from the fact that there exists $C>0$ such that $\jp{\xi} \leq C\jp{\xi_1}\jp{\xi_2}$ and $\jp{
-\xi_1}=\jp{\xi_1}$. Notice that
$$
\jp{\xi_1}^K|\widehat{\varphi}_*(\xi_1) | =\jp{\xi_1}^{-1}\jp{\xi_1}^{K+1}|\widehat{\varphi}_*(\xi_1) |\leq C \jp{\xi_1}^{-1}|\lambda_{\xi_1}|^{K+1}|\widehat{\varphi}_*(\xi_1) | =C\jp{\xi_1}^{-1}\left|\widehat{\tfrac{\mathrm{d}^{K+1}\varphi}{\mathrm{d}x_1^{K+1}}}_*(\xi_1) \right|
$$
Using the fact that $\sum\limits_{\xi_1\in \mathbb{Z}} \jp{\xi_1}^{-2} < \infty$, we obtain
$$
\sum_{\xi_1 \in \mathbb{Z}}\jp{\xi_1}^K|\widehat{\varphi}_* (\xi_1)| \leq C \sum_{\xi_1 \in \mathbb{Z}} \jp{\xi_1}^{-1}\left|\widehat{\tfrac{\mathrm{d}^{K+1}\varphi}{\mathrm{d}x_1^{K+1}}}_*(\xi_1) \right| \leq \left(\sum_{\xi_1 \in \mathbb{Z}} \jp{\xi_1}^{-2}\right)^{1/2} \left(\sum_{\xi_1 \in \mathbb{Z}} \left|\widehat{\tfrac{\mathrm{d}^{K+1}\varphi}{\mathrm{d}x_1^{K+1}}}_*(\xi_1) \right|^2\right)^{1/2}
$$
By the Plancherel inequality, we conclude that 
$$
	\left| 	\jp{\mathcal{F}_2w(\cdot,\xi_2), \varphi} \right| \leq C\jp{\xi_2}^K\left\|\tfrac{\mathrm{d}^{K+1}\varphi}{\mathrm{d}x_1^{K+1}} \right\|_{L^2((0,1))} \leq Cp_{K+1}(\varphi)\jp{\xi_2}^{K+1},
$$
and the proof is completed.
\end{proof}
}
Let $A:[0,1] \to \mathbb{C}$ given by
$$
A(x_1):= \int_0^{x_1} a(s) \mathrm{d} s - x_1a_0,
$$ 
where $a_0=\int\limits_{0}^1 a(s) \mathrm{d} s$. Notice that $A \in C^\infty(\mathbb{T}^1)$ and $A'(x_1) = a(x_1)-a_0$. Define the following operator
$$
\Psi_a w (x_1,x_2):= \sum_{\xi_2 \in \mathbb{Z}} e^{(\log h_2 + 2\pi i \xi_2)A(x_1)}\mathcal{F}_2w(x_1,\xi_2)u_{\xi_2}(x_2).
$$
Notice that $\mathcal{F}_2(\Psi_a w) (x_1,\xi_2) =  e^{(\log h_2 + 2\pi i \xi_2)A(x_1)}\mathcal{F}_2w(x_1,\xi_2).$

\begin{proposition}
	The operator $\Psi_a$ is an automorphism of $H^\infty_{L_h}(\Omega)$ and of $H^{-\infty}_{L_h}(\Omega)$.
\end{proposition}
\begin{proof}
Notice that $(\Psi_a)^{-1} = \Psi_{-a}$, so to complete the proof let us show that $\Psi_a H^\infty_{L_h}(\Omega) \subseteq H^\infty_{L_h}(\Omega)$ and $\Psi_a H^{-\infty}_{L_h}(\Omega) \subseteq H^{-\infty}_{L_h}(\Omega)$. Let $w \in H^\infty_{L_h}(\Omega)$. Let us show that $\mathcal{F}_2(\Psi_a w) (\cdot ,\xi_2)$ satisfies the hypothesis of Theorem \ref{Partial-Func}. For $k \in \mathbb{N}$ fixed we have
\begin{align*}
\dfrac{\mathrm{d}^k}{\mathrm{d}x_1^k}\mathcal{F}_2(\Psi_a w) (x_1 ,\xi_2) &= \dfrac{\mathrm{d}^k}{\mathrm{d}x_1^k}\left(e^{(\log h_2 + 2\pi i \xi_2)A(x_1)}\mathcal{F}_2w(x_1,\xi_2)\right) \\
&= \sum_{\ell=0}^k \binom{k}{\ell} \left(\dfrac{\mathrm{d}^\ell}{\mathrm{d}x_1^\ell}e^{(\log h_2 + 2\pi i \xi_2)A(x_1)}\right)\left(\dfrac{\mathrm{d}^{k-\ell}}{\mathrm{d}x_1^{k-\ell}}\mathcal{F}_2w(x_1,\xi_2)\right).
\end{align*}
By Faà di Bruno's formula, we have
$$
 \dfrac{\mathrm{d}^\ell}{\mathrm{d}x_1^\ell}\left(e^{(\log h_2 + 2\pi i \xi_2)A(x_1)}\right)= \sum_{\gamma \in \Delta(\ell)} \frac{\ell!}{\gamma!} (\log h_2 +2\pi i \xi_2)^{|\gamma|} e^{(\log h_2 + 2\pi i \xi_2)A(x_1)} \prod_{j=1}^\ell \left( \frac{\tfrac{\mathrm{d}^jA(x_1)}{\mathrm{d}x_1^j}}{j!}\right)^{\gamma_j},
$$
where $\Delta(\ell):= \{\gamma \in \mathbb{N}_0^\ell; \sum\limits_{j=1}^\ell j\gamma_j=\ell \}$. Since $A \in C^\infty(\mathbb{T}^1)$, there exists $C_\ell>0$ such that
$$
\left| \dfrac{\mathrm{d}^\ell}{\mathrm{d}x_1^\ell}\left(e^{(\log h_2 + 2\pi i \xi_2)A(x_1)}\right) \right| \leq C_\ell \jp{\xi_2}^\ell,
$$ for all $x_1 \in [0,1]$. Applying Theorem \ref{Partial-Func} for $\mathcal{F}_2w(\cdot, \xi_2)$, for any $N>0$ we obtain $C_{kN}$ such that
\begin{align*}
\left\|\dfrac{\mathrm{d}^k}{\mathrm{d}x_1^k}\mathcal{F}_2(\Psi_a w) (\cdot ,\xi_2)\right\|_{L^2((0,1))} &\leq \sum_{\ell=0}^k\binom{k}{\ell}C_\ell \jp{\xi_2}^\ell \left\|\dfrac{\mathrm{d}^{k-\ell}}{\mathrm{d}x_1^{k-\ell}}\mathcal{F}_2w(\cdot,\xi_2) \right\|_{L^2((0,1))}\\
&\leq \sum_{\ell=0}^k\binom{k}{\ell}C_\ell \jp{\xi_2}^\ell C_{N\ell}\jp{\xi_2}^{-(N+\ell)}\\
& \leq C_{kN}\jp{\xi_2}^{-N}
\end{align*}
Therefore $\Psi_a w \in H^{\infty}_{L_h}(\Omega)$. Similarly, we obtain that $\Psi_a$ is an automorphism of $H^{-\infty}_{L_h}(\Omega)$  and the proof is omitted.
\end{proof}
Consider the operator $P_0$ given by
$$
P_0:= \frac{\partial }{ \partial x_1}+a_0\frac{\partial}{\partial x_2}.
$$

\begin{theorem}\label{normal}
We have that
$$
\Psi_a \circ P = P_0 \circ \Psi_a.
$$
Moreover, the operator $P$ is globally $L_h$--hypoelliptic if and only if the operator $P_0$ is globally $L_h$--hypoelliptic.
\end{theorem}
\begin{proof}
	For every $w\in H^\infty_L(\Omega)$ and $\xi_2\in\mathbb{Z}$ we have
	$$
	\mathcal{F}_2(\Psi_a (Pw))(x_1,\xi_2) = \mathcal{F}_2(P_0(\Psi_a w))(x_1, \xi_2).
	$$
	Indeed, for all $\xi_2 \in \mathbb{Z}$ we have
	\begin{align*}
	 \mathcal{F}_2(P_0(\Psi_a w))(x_1, \xi_2) &=  \mathcal{F}_2\left(\frac{\partial }{ \partial x_1}\Psi_a w+a_0\frac{\partial}{\partial x_2}\Psi_a w\right)(x_1, \xi_2)\\
	 &=\frac{\mathrm{d} }{ \mathrm{d} x_1} \mathcal{F}_2(\Psi_a w)(x_1, \xi_2)+a_0(\log h_2+2\pi i\xi_2) \mathcal{F}_2(\Psi_a w)(x_1, \xi_2)\\
	 &=\frac{\mathrm{d} }{ \mathrm{d} x_1} \left(e^{(\log h_2 + 2\pi i \xi_2)A(x_1)}\mathcal{F}_2w(x_1,\xi_2)\right) \\&+ a_0(\log h_2+2\pi i\xi_2)e^{(\log h_2 + 2\pi i \xi_2)A(x_1)}\mathcal{F}_2w(x_1,\xi_2)\\
	 &=(a(x_1)-a_0)(\log h_2+2\pi i\xi_2)e^{(\log h_2 + 2\pi i \xi_2)A(x_1)}\mathcal{F}_2w(x_1,\xi_2) \\
	 &+ e^{(\log h_2 + 2\pi i \xi_2)A(x_1)}\frac{\mathrm{d} }{ \mathrm{d} x_1}\mathcal{F}_2w(x_1,\xi_2)\\&+ a_0(\log h_2+2\pi i\xi_2)e^{(\log h_2 + 2\pi i \xi_2)A(x_1)}\mathcal{F}_2w(x_1,\xi_2)\\
	 &=e^{(\log h_2 + 2\pi i \xi_2)A(x_1)} \mathcal{F}_2 \left(\frac{\partial }{ \partial x_1} w+a(x_1)\frac{\partial}{\partial x_2} w\right)(x_1, \xi_2)\\
	 &=	\mathcal{F}_2(\Psi_a (Pw))(x_1,\xi_2).
	\end{align*}

	Suppose that $P$ is globally ${L_h}$--hypoelliptic. If $P_0w = f \in H_{L_h}^{\infty}(\Omega)$ for some $u\in H^{-\infty}_{L_h}(\Omega)$, then
	$\Psi_{-a}P_0w = \Psi_{-a}f \in H^\infty_{L_h}(\Omega)$. Since $ \Psi_{-a} \circ P_0=P \circ \Psi_{-a},$ we have
	$P(\Psi_{-a} w) \in H^\infty_{L_h}(\Omega)$ and by global ${L_h}$--hypoellipticity of $P$ we have $\Psi_{-a}w \in H_{L_h}^\infty(\Omega)$, which implies that $w\in H_{L_h}^\infty(\Omega)$ and then $P_0$ is globally ${L_h}$--hypoelliptic.
	
	Assume now that $P_0$ is globally ${L_h}$--hypoelliptic. If $Pw = f \in H_{L_h}^\infty(\Omega)$ for some $w\in H^{-\infty}_{L_h}(\Omega)$, we can write $P (\Psi_{-a} \Psi_{a} w) = f \in H_{L_h}^\infty(\Omega)$. By the fact that $P \circ \Psi_{-a} = \Psi_{-a} \circ P_0 $ we obtain $\Psi_{-a} P_0(\Psi_{a} w) = f$, that is, $P_0(\Psi_{a} w) = \Psi_{a} f \in H_{L_h}^\infty(\Omega)$. By the global ${L_h}$--hypoellipticity of $P_0$ we have that $\Psi_{a} w \in H_{L_h}^\infty(\Omega)$ and then $w \in H_{L_h}^\infty(\Omega)$.
\end{proof}
Finally, let us investigate the ${L_h}$--solvability of the operator $P$. We say that the operator $P$ is globally $L_h$--solvable if $PH^{\infty}_{L_h}=\mathbb{F}$, where
$$
\mathbb{F}=\{f \in H^{-\infty}_{L_h}(\Omega); \Psi_{a} f \in \mathbb{E}_0\},
$$ 
and $\mathbb{E}_0$ is the space of $P_0$--admissible $L_h$--distributions as in \eqref{adm}.
\begin{theorem}
	The operator $P$ is globally $L_h$--solvable if and only if $P_0$ is globally $L_h$--solvable.
\end{theorem}
\begin{proof}
		Suppose that $P$ is globally $L_h$--solvable and let $f \in \mathbb{E}$ a $P_0$--admissible $L_h$--distribution. Hence, $\Psi_{-a}f \in \mathbb{F}$ and so there exists $w \in H^{-\infty}_{L_h}(\Omega)$ such that $Pw=\Psi_{-a}f$, which implies that $\Psi_aPw=f$. By Theorem \ref{normal}, we conclude that $P_0(\Psi_a w) = f$. Hence $P_0$ is globally $L_h$--solvable.
		
		Assume now that $P_0$ is globally $L_h$--solvable and let $f \in \mathbb{F}$, that is, $\Psi_a f \in \mathbb{E}$. Thus, there exists $w \in H^{-\infty}_{L_h}(\Omega)$ such that $P_0w=\Psi_a f$. Again by Theorem \ref{normal}, we have $f = \Psi_{-a} P_0 w = P (\Psi_{-a} w)$ and we conclude that $P$ is globally $L_h$--solvable.
\end{proof}
\section*{Acknowledgments}
The author would like to thank Niyaz Tokmagambetov for his comments and suggestions. 
The author is supported by the Methusalem programme of the Ghent University Special Research Fund (BOF) (Grant number 01M01021). 
\bibliographystyle{spmpsci}      
\bibliography{biblio}

\end{document}